\documentclass[11pt]{article}
\usepackage{mathrsfs}
\usepackage{bbm}

\usepackage{verbatim}

\usepackage{amsthm}
\usepackage{amsmath}
\usepackage{amssymb}

\usepackage{graphicx}
\usepackage{epstopdf}
\usepackage{titling}

\usepackage{enumerate}

\usepackage{algpseudocode}
\usepackage{algorithm}
\usepackage{algorithmicx}
\usepackage{tikz}

\newtheorem{theorem}{Theorem}[section]

\newtheorem{lemma}[theorem]{Lemma}

\newtheorem{corollary}[theorem]{Corollary}

\newtheorem{example}[theorem]{Example}

\newtheorem{definition}[theorem]{Definition}

\newcommand{\I}{[0,1]}

\newcommand{\G}{\mathcal{G}}
\newcommand{\uE}[1]{\underset{#1}{\mathbb{E}}}
\newcommand{\E}{\mathbb{E}}

\begin{document}

\title{On the 3-local profiles of graphs}

\author{Hao Huang\thanks{
School of Mathematics, Institute for Advanced Study, Princeton 08540. Email: huanghao@math.ias.edu. Research
supported in part by NSF grant DMS-1128155.
}
\and Nati Linial\thanks{School of Computer Science and engineering, The Hebrew University of Jerusalem, Jerusalem 91904, Israel. Email {\tt nati@cs.huji.ac.il}. Research supported in part by the Israel Science Foundation and by a USA-Israel BSF grant.}
\and Humberto Naves\thanks{Department of Mathematics, UCLA, Los Angeles, CA 90095. Email: {\tt hnaves@math.ucla.edu}.}
\and Yuval Peled\thanks{School of Computer Science and engineering, The Hebrew University of Jerusalem, Jerusalem 91904, Israel. Email {\tt yuvalp@cs.huji.ac.il}}
\and Benny Sudakov\thanks{Department of Mathematics, ETH, 8092 Zurich, Switzerland and Department of Mathematics, UCLA, Los Angeles, CA 90095. Email: 
bsudakov@math.ucla.edu. 
Research supported in part by NSF grant DMS-1101185, by AFOSR MURI grant FA9550-10-1-0569 and by a USA-Israel BSF
grant.
}}

\date{}

\maketitle

\setcounter{page}{1}

\vspace{-2em}
\begin{abstract}
For a graph $G$, let $p_i(G), i=0,...,3$ be the probability that three distinct random vertices
span exactly $i$ edges. We call $(p_0(G),...,p_3(G))$ the {\em 3-local profile}
of $G$.
We investigate the set ${\cal S}_3 \subset \mathbb R^4$ of all vectors $(p_0,...,p_3)$ that are arbitrarily close to the 3-local profiles of arbitrarily large graphs. We give a full description of the projection of ${\cal S}_3$ to the $(p_0, p_3)$ plane. The upper envelope of this planar domain is obtained from cliques on a fraction of the vertex set and complements of such graphs. The lower envelope is Goodman's inequality $p_0+p_3\geq\frac{1}{4}$. We also give a full description of the triangle-free case, i.e., the intersection of ${\cal S}_3$ with the hyperplane $p_3=0$. This planar domain is characterized by an SDP constraint that is derived from Razborov's flag algebra theory.
\end{abstract}
\section{Introduction}
\label{section_intro}
For graphs $H,G$, we denote by $d(H;G)$ the induced density of the graph $H$ in the graph $G$. Namely, the probability that a random set of $|H|$ vertices in $G$ induces a copy of the
graph $H$.

Many important problems and theorems in graph theory can be formulated in the framework of graph densities. Most of the emphasis so far has been on edge counts, or equivalently,
on maximizing $d(K_2;G)$ subject to some restrictions. Thus Tur\'an's theorem determines $\max d(K_2;G)$ under the assumption $d(K_s;G)=0$ for some $s\ge 3$. The theorem further says
that the optimal graph is the complete balanced $(s-1)$-partite graph. This was substantially extended by Erd\H{o}s and Stone \cite{erdos-stone} who determined $\max d(K_2; G)$ under
the assumption that the $H$-density (not induced) of $G$ is zero for some fixed graph $H$. Their theorem also shows that the answer depends only on the chromatic number of $H$.
Ramsey's theorem shows that for any two integers
$r, s \ge 2$, every sufficiently large graph $G$ has either $d(K_s, G)>0$ or $d(\overline{K_r}, G)>0$. The Kruskal-Katona Theorem \cite{katona, kruskal}, can be stated as saying that
$d(K_r; G)=\alpha$ implies that $d(K_s; G) \le \alpha^{s/r}$ for $r \leq s$. Finding $\min d(K_s; G)$ under the assumption $d(K_r; G)=\alpha$ turns out to be more difficult. The case
$r=2$ of this problem was solved only recently in a series of papers by Razborov \cite{razborov}, Nikiforov \cite{nikiforov} and Reiher \cite{reiher}. A closely related question is to
minimize $d(K_s; G)$ given that $d(\overline{K}_r; G) = \alpha$ for some real $\alpha \in [0,1]$ and integers $r, s \ge 2$. The case $\alpha=0$ of this problem was
posed by Erd\H{o}s more than 50 years ago. Although, recently Das et al \cite{das-et-al}, and independently Pikhurko
\cite{pikhurko}, solved it for certain values of $r$ and $s$ it is still open in general.

Numerous further questions concerning the numbers $d(H;G)$ suggest themselves. Goodman~\cite{goodman} showed that $\min_G d(K_3; G) + d(\overline{K}_3; G) = 1/4-o(1)$, and the random graph
$G(n, \frac 12)$ shows that this bound is tight. Erd\H{o}s \cite{erdos-false} conjectured that a $G(n, \frac 12)$ graph also minimizes $d(K_r; G) + d(\overline{K}_r; G)$ for all $r$,
but this was refuted by Thomason \cite{thomason} for all $r \ge 4$. A simple consequence of Goodman's inequality is that $\min_G \max \{d(K_3; G), d(\overline{K}_3; G)\}=1/8$. The
analogous statement for $r=4$ is not true as can be shown using an example of Franek and R\"odl \cite{franek-rodl} (see \cite{cl_ac} for the details). On the other hand, the max-min
version
of this problem is now solved. As we have recently
proved~\cite{cl_ac}, $\max_G \min \{d(K_r; G), d(\overline{K}_r; G)\}$ is obtained by a clique on a properly chosen fraction of the vertices.

Closely related to these questions is the notion of inducibility of graphs, first introduced in \cite{pippenger1975inducibility}. The inducibility of a graph $H$ is defined as $\lim_{n
\to \infty} \max_G d(H;G)$, where the maximum is over all $n$-vertex graphs $G$. This natural parameter has been investigated for several types of graphs $H$. E.g., complete bipartite
and multipartite graphs \cite{bollobas1986maximal, brown1994inducibility}, very small graphs \cite{exoo1986dense, hirst2011inducibility} and blow-up graphs
\cite{hatami2011inducibility}.

In light of this discussion, the following general concept suggests itself.
\begin{definition}\label{qstn_general} For a family of finite graphs $\mathcal{H}=(H_1,...,H_t)$, let $d(\mathcal{H}; G):=(d(H_1; G),\ldots d(H_t;G))$. Define
$\Delta(\mathcal H)$ to be the set of all $\bar{p}=(p_1,...,p_t)\in\I^t$ for which there exists a sequence of graphs $G_n$, such that $|G_n|\to\infty$ and $d(\mathcal H;G_n)\to\bar p$.
We likewise define $\Delta_\mathcal{G}(\mathcal H)$ where we require that $G_n\in\mathcal{G}$, an infinite families of graphs of interest (e.g. $K_s$-free graphs).
\end{definition}

The initial discussion suggests that it may be a very difficult task to fully describe $\Delta(\mathcal H)$ or $\Delta_\mathcal{G}(\mathcal H)$. Indeed, it was shown by Hatami and
Norine \cite{undecide} that in general it is undecidable to determine the linear inequalities that such sets satisfy. In this paper we solve two instances of this question.

We denote by $p_i(G)$ the probability that three distinct random vertices in the graph $G$ span exactly $i$ edges. The first theorem describes the possible distributions of 3-cliques
and 3-anticliques in graphs (i.e., of $(p_0,p_3)$). We have Goodman's inequality~\cite{goodman} as a lower bound, and an upper bound from \cite{cl_ac}. We show that these bounds fully describe all possible $(p_0,p_3)$.
\begin{theorem}\label{theorem_cliques}
For $p_0\in\I$, let $\beta$ be the unique root in $[0,1]$ of $\beta^3+3\beta^{2}(1-\beta)={p_0}$. Then, $(p_0,p_3)\in\Delta(\overline{K_3},K_3)$ iff \[p_0+p_3\geq\frac{1}{4}\;\;and\;\; p_3\leq\max \{(1-{p_0}^{1/3})^3 + 3{p_0}^{1/3}(1-{p_0}^{1/3})^{2}, (1-\beta)^3\}.\]
\end{theorem}
The analogous question concerning $\Delta(\overline{K_r},K_r)$ for $r>3$ is widely open. While the analogous upper bound is proved in~\cite{cl_ac}, the situation with respect to the lower bounds is still poorly understood~\cite{giraud1979probleme, sperfeld2011minimal}.

The second theorem in this paper is proved using the theory of flag algebras \cite{raz_fa}. This theory provides a method to derive upper bounds in asymptotic extremal graph theory. This is accomplished by generating certain semidefinite programs (=SDP) that pertain to the problem at hand. By passing to the dual SDP we derive necessary conditions for membership in $\Delta(\mathcal H)$ or $\Delta_{\mathcal G}(\mathcal H)$. Section \ref{section_fa} contains a self contained discussion, covering this perspective of the theory of flag algebras.

The theorem below demonstrates the special role that bipartite graphs play in the study of triangle-free graphs. As the theorem shows, all 3-local profiles of triangle-free graphs are also realizable by bipartite graph. Moreover, the theory of flag algebras provides a complete answer to this question. This yields a different perspective to the fact that almost all triangle-free graphs are bipartite \cite{er_kl_ro}. We denote the $3$-vertex path by $P_3$ and its complement by $\overline{P_3}$. Also, as usual, $A \succeq 0$ means that the matrix $A$ is positive semi-definite (=PSD). The class of bipartite (resp. triangle-free) graphs is denoted by $\mathcal{BP}$ (resp. $\mathcal{TF}$).
\begin{theorem}\label{theorem_tr_free}
For $p_0,p_1,p_2 \ge 0$ s.t. $p_0+p_1+p_2=1$, the following conditions are equivalent:
\begin{enumerate}[I.]
\item $(p_0,p_1,p_2)\in \Delta_{\mathcal{TF}}(\overline{K_3},\overline{P_3},P_3)$
\item $p_0\left(\begin{matrix}
1&0\\0&0
\end{matrix}\right)+
p_1\left(\begin{matrix}
\frac{1}{3}&\frac{1}{3}\\\frac{1}{3}&0
\end{matrix}\right)+
p_2\left(\begin{matrix}
0&\frac{1}{3}\\\frac{1}{3}&\frac{1}{3}
\end{matrix}\right)\succeq 0$
\item $(p_0,p_1,p_2)\in \Delta_{\mathcal{BP}}(\overline{K_3},\overline{P_3},P_3)$
\end{enumerate}
\end{theorem}
The remainder of this paper is organized as following. In Section \ref{section_rand} we use random graphs to show the realizability of $\Delta_\mathcal{G}(\mathcal H)$. In section \ref{section_cliques} we prove Theorem \ref{theorem_cliques}. In Section \ref{section_fa} we use the theory of flag algebras to derive SDP constraints on membership in $\Delta_\mathcal{G}(\mathcal H)$. In section \ref{section_tf} we prove Theorem \ref{theorem_tr_free}. We close with some concluding remarks and several open problems.
\section{Random Constructions}
\label{section_rand}
Let $\mathcal H=(H_1,...,H_t)$ be a collection of graphs, and $\bar p =(p_1,...,p_t)\in\I^t$. In order to prove that $\bar p\in\Delta_\mathcal{G}(\mathcal H)$, we need arbitrarily large graphs $G$ for which $\|d(\mathcal{H};G)-\bar p\|$ is negligible. We accomplish this using appropriately designed random $G$'s. Let $\Pi$ be a symmetric $n \times n$ matrix with entries in $\I$ and zeros along the diagonal. Corresponding to $\Pi$ is a distribution $G(\Pi)$ on $n$-vertex graphs where $ij$ is an edge with probability $\Pi_{i,j}$ and the choices are made independently for all $n \ge i > j \ge 1$. We say that a graph $G$ is {\em supported} on $\Pi$ if $G$ is chosen from $G(\Pi)$ with positive probability.

\begin{lemma}\label{lemma:rand_graph}
For every list of graphs $H_1,...,H_t$ there exists an integer $N_0$ such that if $n>N_0$ and $\Pi$ is an $n \times n$ matrix as above, then there exists an $n$-vertex graph $G^{\ast}$ supported on $\Pi$ such that
\[
\forall\;i=1,...,t\quad \left|d(H_i;G^{\ast})-\uE{G\sim G(\Pi)}[d(H_i;G)]\right|\leq\frac{1}{\sqrt{n}}
\]
\end{lemma}

We note that the statement need not hold if $G(\Pi)$ is replaced by an arbitrary distribution on $n$-vertex graphs.

\begin{proof}
Fix a graph $H$. Let us view an $n$-vertex graph $G$ as a ${\binom{n}{2}}$-dimensional binary vector. The mapping $G \mapsto d(H;G)$ has Lipschitz constant
$\binom{|H|}{2}/\binom{n}{2}$. We can therefore apply Azuma's inequality and conclude that
\[
\Pr_{G^{\ast}\sim G(\Pi)}\left[\left|d(H;G^{\ast})-\uE{G\sim G(\Pi)}[d(H;G)]\right|\geq\frac{1}{\sqrt{n}}\right]\leq 2\exp\left(-{\frac{\binom{n}{2}}{2n\binom{|H|}{2}^2}}\right).
\]
Using the union bound and denoting $h=\max|H_i|$, we get
\[
\Pr_{G^{\ast}\sim G(\Pi)}\left[\left\|d({\cal H};G^{\ast})-\uE{G\sim G(\Pi)}[d({\cal H};G)]\right\|_{\infty}\leq\frac{1}{\sqrt{n}}\right]\geq
1-2t\exp\left(-{\frac{\binom{n}{2}}{2n\binom{h}{2}^2}}\right)=1-o_n(1).
\]
\end{proof}
This lemma is easily generalized for hypergraphs of greater uniformity.

\section{Distribution of 3-cliques and 3-anticliques}
\label{section_cliques}
In this section we prove Theorem \ref{theorem_cliques} and produce a full description of the set $\Delta(\overline{K_3},K_3)$. We state the known lower and upper bounds and show that they fully describe this set.
\begin{theorem}[Goodman \cite{goodman}]
\label{theorem_goodman}
For every $n$-vertex graph $G$
\[
p_0(G)+p_3(G)\geq\frac{1}{4}-O\left(\frac 1n\right).
\]
\end{theorem}
\begin{theorem}[\cite{cl_ac}]
\label{theorem_clac}
Let $r, s \ge 2$ be integers and suppose that $d(\overline{K}_r; G) \geq \alpha$ where $G$ is an $n$-vertex graph and $1 \ge \alpha \ge 0$. Let $\beta$ be the unique
root of $\beta^r+r\beta^{r-1}(1-\beta)=\alpha$ in $[0,1]$. Then $$d(K_s; G) \le \max \{(1-\alpha^{1/r})^s + s\alpha^{1/r}(1-\alpha^{1/r})^{s-1},
(1-\beta)^s\}+o(1)$$Namely, given $d(\overline{K}_r; G)$, the maximum of $d(K_s; G)$ is attained up to a negligible error-term either by a clique on some subset of the $n$
vertices, or by the complement of such a graph. In particular, for every $G$
\[
p_3(G)\leq\max \{(1-{p_0(G)}^{1/3})^3 + 3{p_0(G)}^{1/3}(1-{p_0(G)}^{1/3})^{2}, (1-\beta)^3\},
\]
where $\beta$ is the unique root of $\beta^3+3\beta^{2}(1-\beta)={p_0(G)}$ in $[0,1]$.
\end{theorem}

\begin{figure}[h!]
  \centering
    \includegraphics[width=0.5\textwidth]{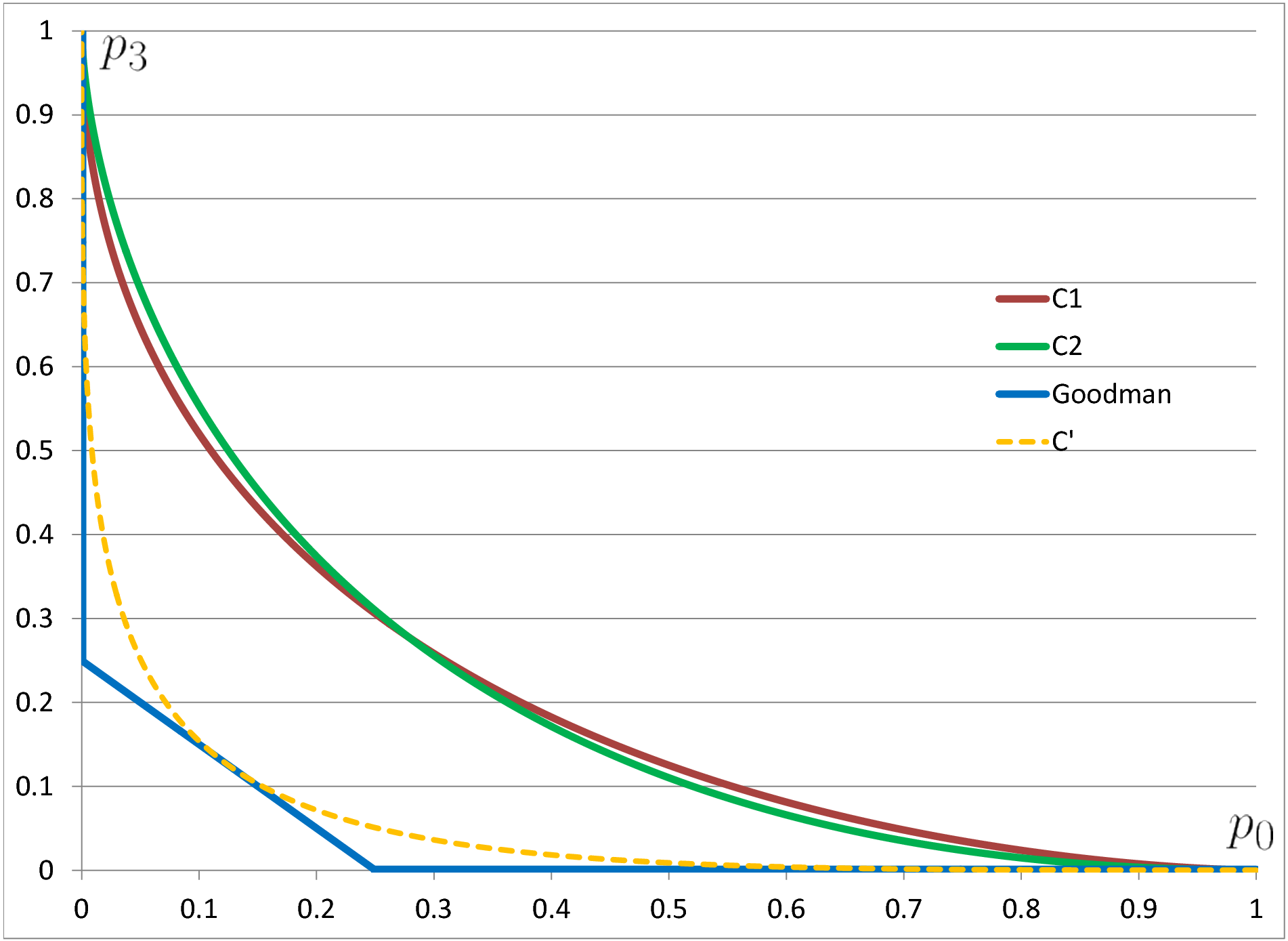}
    \caption{The curves that bound the region of $\Delta (\overline{K_3},K_3)$, and the auxiliary curve $C'$ used in the proof}
\end{figure}

\begin{proof}[Proof of Theorem \ref{theorem_cliques}]
Let $C_1,C_2$ be the $(p_0,p_3)$ curves induced by cliques and complements of cliques resp.
\begin{eqnarray*}
C_1=\left\{\left( (1-x)^3+3(1-x)^2x,x^3\right)~|~x\in\I\right\}\\
C_2=\left\{\left( x^3,(1-x)^3+3(1-x)^2x \right)~|~x\in\I\right\}
\end{eqnarray*}
For $i=1, 2$ let $B_i\subset\I^2$ be the region bounded by $p_0\geq 0,\;p_3\geq 0,\;p_0+p_3\geq\frac{1}{4},$ and by $C_i$. We need to prove that
\[
\Delta (\overline{K_3},K_3) = B_1\cup B_2.
\]
By Theorems \ref{theorem_goodman} and \ref{theorem_clac} $\Delta\subseteq B_1\cup B_2$.

We show that every point in this domain can be approximated arbitrarily well by $(p_0(G),p_3(G))$ for arbitrarily large $G$. We define the following parameterized family of random graphs:
\begin{definition} \label{def_Gxabc}
For every $x,a,b,c\in \I$, $G_{x,a,b,c}$, is the class of random graphs $(V,E)$, where $V = A \dot \cup B$ with $|A| =x|V|$ and $|B|=(1-x)|V|$. Adjacencies are chosen independently among pairs, with
\[
Pr(ij \in E) = \left\{\begin{matrix}
a & i,j\in A \\
b & i,j\in B \\
c & i\in A,\;j\in B\mbox{ or vice versa}\\
\end{matrix}\right.
\]
\end{definition}
A simple computation shows that
\[\E p_0(G_{x,a,b,c}) =
x^3(1-a)^3+(1-x)^3(1-b)^3+3x^2(1-x)(1-a)(1-c)^2+3x(1-x)^2(1-b)(1-c)^2 + o(1)\]
and
\[ \E p_3(G_{x,a,b,c}) = x^3a^3+(1-x)^3b^3+3x^2(1-x)ac^2+3x(1-x)^2bc^2+o(1)\]
By Lemma \ref{lemma:rand_graph},  $(\E p_0(G_{x,a,b,c}),\E p_3(G_{x,a,b,c}))\in\Delta (\overline{K_3},K_3)$ for every $(x,a,b,c)\in\I^4$.
The following curve is used in the proof.
\[
C'=\left\{\left( t^3,(1-t)^3\right)~|~t\in\I\right\}
\]
Consider the following continuous map,
\[
H:\I\times\I\rightarrow\Delta (\overline{K_3},K_3)
\]
\[H(x,a)=\left(\mathbb{E}[p_0(G_{x,a,1-a,1-a})],\mathbb{E}[p_3(G_{x,a,1-a,1-a})]\right).\]
The following claims are immediate.
\begin{enumerate}
\item $H(x,0)=\left( x^3,(1-x)^3+3(1-x)^2x \right)$.
\item $H(x,1)=\left( (1-x)^3+3(1-x)^2x,x^3\right)$
\item $H(1,a)=((1-a)^3,a^3)$
\item $H(x,\frac{1}{2})=(\frac{1}{8},\frac{1}{8})$.
\item $H(0,a)=(a^3,(1-a)^3)$

\end{enumerate}
$H\mid_{\I\times\left[0,\frac{1}{2}\right]}$ is a continuous map from a topological $2$-disc. The boundary of this disk is mapped to a path encircling  $C_2\cup C'$. Therefore, $C_2\cup C'$ is contractible in $im(H)$, and consequently the region bounded by $C_2$ and $C'$ is contained in $im(H)$, and also in $\Delta (\overline{K_3},K_3)$. A similar argument for $H\mid_{\I\times\left[\frac{1}{2},1\right]}$ shows that the region bounded by $C_1$ and $C'$ is contained in $\Delta (\overline{K_3},K_3)$.\\
The remaining area in $\Delta (\overline{K_3},K_3)$ will be covered similarly. Consider the following continuous map,
\[
H_1:\I\times\I\rightarrow\Delta (\overline{K_3},K_3)
\]
\[H_1(x,a)=\left(\E[p_0(G_{x,a,a,1-a})],\E[p_3(G_{x,a,a,1-a})]\right).\]
Again, the following claims are immediate.
\begin{enumerate}
\item $H_1(x,0)=(x^3+(1-x)^3,0)$
\item $H_1(\frac{1}{2},a)=\frac{1}{8}\left(1-(2a-1)^3,1+(2a-1)^3\right)$
\item $H_1(x,1)=(0,x^3+(1-x)^3)$
\item $H_1(0,a)=((1-a)^3,a^3)$

\end{enumerate}
$H_1\mid_{\left[0,\frac{1}{2}\right]\times\I}$ is a continuous map from a topological $2$-disc, mapping its boundary to a path encircling $C',\;[\frac{1}{4},1]\times\{0\},\{(\frac{t}{4},\frac{1-t}{4})\mid t\in\I\}$ and $\{0\}\times[\frac{1}{4},1]$. Therefore, as before, the region between these curves is contained in $\Delta (\overline{K_3},K_3)$. Altogether, $B_1 \cup B_2\subseteq\Delta (\overline{K_3},K_3)$ is obtained.

\end{proof}

\section{Flag algebras - a dual perspective}
\label{section_fa}
Let $\G$ be an infinite family of graphs closed under taking induced subgraphs, let $\mathcal H =(H_1,...,H_t)$ a collection of graphs. We formulate necessary conditions for membership in the set $\Delta_\G(\mathcal H)$ which are stated in terms of feasibility of some SDP. This part is self-contained, and concentrates on the connections between the theory of flag algebras and standard arguments in discrete optimization.
\begin{definition}
An $(s,k)$-\emph{flagged graph} $F=(H,U)$ consists of an $s$-vertex graph $H$ and a {\em flag} $U=(u_1,...,u_k)$, an \emph{ordered} set of $k$ vertices in $H$. An isomorphism $F\cong
F'$ between flagged graphs $F=(H,U)$ and $F'=(H',U')$ is a graph isomorphism
$$
\varphi:V(H)\longrightarrow V(H')\mbox{ such that }\varphi(u_i)=u'_i\quad\forall i.
$$
\end{definition}

\begin{definition}\label{def:prob_flag}
Let $G$ be a graph and $F_1,F_2$ be $(s,k)$-flagged graphs.
Choose uniformly at random two subsets $(V_1,V_2)$ of $V(G)$ of size $s$ with intersection $U = V_1 \cap V_2$ of cardinality $k$ and choose a random ordering of $U$. Define
\[
p(F_1,F_2;G)=\Pr\left[F_i \cong (G|_{V_i},U)\quad i=1,2\right].
\]
Associated with every list $F_1,...,F_l$ of $(s,k)$-flagged graphs is the $l \times l$ matrix $A^G=A^G(F_1,...,F_l)$
\[
\forall i,j\quad A^G(F_1,...,F_l)_{i,j}=p(F_i,F_j;G).
\]
\end{definition}
Note that $A^G$ is a symmetric matrix.
\begin{example}\label{exmp:sfmat}
Denote by $e$ (resp. $\overline{e}$) the edge (its complement) with one flagged vertex. Also, $P_3$ denotes the path on 3 vertices. Then
\[
A^{P_3}(\bar{e},e)=\left(\begin{matrix}
0&\frac{1}{3}\\
\frac{1}{3}&\frac{1}{3}
\end{matrix}\right)
\]
\begin{proof}
Let $V_1,V_2\subset V(P_3)$ be chosen randomly with $|V_1|=|V_2|=2$ and $|V_1\cap V_2|=1$. First, $p(\bar{e},\bar{e};P_3)=0$ since either $V_1$ or $V_2$ spans an edge. Also, $ p( e,e;P_3)=\frac 13$ since both sets span an edge iff their common vertex has degree 2. Finally,  $p(\bar e,e;P_3)=\frac 13$ since the common vertex has degree 1 with probability 2/3, and conditioned on that, the first set $V_1$ spans an edge with probability 1/2.
\end{proof}
\end{example}
We denote by $PSD(l)$ the cone of $l\times l$ positive semi-definite matrices. The following theorem is an analogue of Theorem 3.3 in \cite{raz_fa}.
\begin{theorem}\label{main_th_fa}
Let  $F_i,\;i=1,...,l$, be $(s,k)$-flagged graphs. For an $n$-vertex graph $G$,
\[
\mbox{dist}\left(A^G(F_1,...,F_l),PSD(l)\right)=O\left(\frac{1}{n}\right)
\]
where $\mbox{dist}$ stands for distance in $l_2$.
\end{theorem}
\begin{corollary}\label{main_cor}
Let $\G$ be a class of graphs that is closed under taking induced subgraphs. Let $\G_n$ be the set of $n$-vertex members of $\G$. Let $\mathcal H=(H_1,...,H_t)$ be a complete list of
all the isomorphism types of graphs in $\G_r$. Let $F_i,\;i=1,...,l$ be $(s,k)$-flagged graphs.
Then for every $(p_1,...,p_t)\in\Delta_{\G}(\mathcal H)$,
\[
\sum_{\alpha=1}^{t}p_\alpha\cdot A^{H_\alpha}(F_1,...,F_l)\succeq 0\,.
\]
\end{corollary}
Let us illustrate how this corollary helps us derive an upper bound on $\lim_{n\to\infty}\max_{G\in\G_n}d(H;G)$ for some fixed graph $H$.
(The limit exists since $\max_{G\in\G_n}d(H;G)$ is a non-increasing function of $n$). Note that
\[
d(H;G)=\sum_{\alpha=1}^{t}d(H;H_\alpha)d(H_\alpha;G).
\]
Therefore the following SDP yields an upper bound
\begin{equation*}\label{eqn:the_prog}
\max\quad \sum_{\alpha=1}^{t}d(H;H_\alpha)p_\alpha\quad s.t.
\end{equation*}
\[
\mbox{all~} p_\alpha\geq 0\mbox{~and~}
\sum_\alpha p_\alpha=1
\]\[
\sum_{\alpha=1}^{t}p_{\alpha}\cdot A^{H_\alpha}(F_1,...,F_l)\succeq 0\]
By SDP duality, this maximum can be also upper-bounded by
\[
\min_{Q\in PSD(l)}\left(\max_{1\leq \alpha \leq t}\left[d(H;H_\alpha)+\mbox{Tr}(Q\cdot A^{H_\alpha}) \right]\right)
\]
which is the more familiar form of SDP used in the literature on applications of flag algebras (see, e.g., \cite{raz_forb, das-et-al} ).
The proofs of the above two statements are based on standard arguments.
\begin{proof}[Theorem \ref{main_th_fa} $\implies$ Corollary \ref{main_cor}]
First we prove that for every graph $G$ on at least $r$ vertices,
\[ A^G(F_1,...,F_l)=\sum_{\alpha=1}^{t}d(H_\alpha;G)A^{H_\alpha}(F_1,...,F_l). \]
Namely, that for every $1\le i,j\le l$,
\[
p(F_i,F_j;G)=\sum_{\alpha=1}^{t}d(H_\alpha;G)p(F_i,F_j;H_\alpha).
\]
This is just an application of the law of total probability. On the LHS we sample uniformly two sets $V_1,V_2$ of size $s$ with $|V_1\cap V_2|=k$ from $V(G)$ together with a random
ordering of $V_1\cap V_2$, and on the RHS we first
sample a random set $V'$ of size $r$ from $V(G)$, and then uniformly sample $V_1,V_2$ as above from $V'$.
To finish the proof, let $\bar p=(p_1,...,p_t)\in\Delta_\G(\mathcal H)$. By the definition of $\Delta_\G(\mathcal H)$ and Theorem~\ref{main_th_fa}, for every $\epsilon>0$ there is a
sufficiently large graph $G \in \G$ such that both
$$
|p_\alpha - d(H_\alpha;G)|<\epsilon\quad\forall\alpha
$$
and
$$
\mbox{dist}\left(\sum_\alpha d(H_\alpha;G)A^{H_\alpha},PSD(l)\right)<\epsilon.
$$
Therefore, $\mbox{dist}\left(\sum_\alpha p_\alpha A^{H_\alpha},PSD(l)\right)=0$.
\end{proof}

\begin{proof}[Proof of Theorem \ref{main_th_fa}]
Let $G$ be an $n$-vertex graph. Consider the following equivalent description of the underlying distribution in the definition of the matrix $A^G=A^G(F_1,...,F_l)$. Choose uniformly at random an ordered set $U\subset V(G)$ of size $k$, two \emph{disjoint} sets $S_1,S_2\subset V(G) \setminus U$ of size $s-k$ and let $V_i=S_i\dot\cup U$, $i=1,2$. Thus $A^G_{i,j}$ is the probability that $F_i \cong (G|_{V_i},U)$, for $i=1,2$.
Note that for every fixed $U$, two sets $S_1,S_2\subset V(G) \setminus U$ of size ${s-k}$ chosen uniformly and independently at random are disjoint with probability $1-O(1/n)$. Therefore, it suffices to prove that the matrix $B^G$, defined exactly as $A^G$ except that $S_1,S_2$ are chosen \emph{independently}, is PSD.

Consider the matrix $Q$ with $l$ rows and $\frac{n!}{(n-k)!}$ columns indexed by ordered sets $U\subset V(G)$ of size $k$, defined as following. Choose a random subset $S\subset V(G)\setminus U$ of size $s-k$, and let  $Q_{i,U}=\Pr[F_i\cong (G|_{S\cup U},U)]$. Then,
\[
B^G=\frac{(n-k)!}{n!}QQ^T\succeq 0
\]
\end{proof}

\section{Triangle-free graphs}
\label{section_tf}
In this section we prove Theorem \ref{theorem_tr_free}, by showing that the set
$
\Delta_{\mathcal{TF}}(\overline{K_3},\overline{P_3},P_3)
$
is characterized by the quadratic constraints deduced from the flag algebra theory.
\begin{proof}[Proof of Theorem \ref{theorem_tr_free}]
$(I)\implies (II).$ This implication is a direct application of Corollary \ref{main_cor}. Let $\overline{e}$,$e$ be (2,1)-flagged graphs. $\overline{e}$ (resp. $e$) is the empty (complete) graph over 2 vertices with one flagged vertex. By a straightforward computation (See example \ref{exmp:sfmat}),
\[
A^{\overline{K_3}}(\bar{e},e)=\left(\begin{matrix}1&0\\0&0\end{matrix}\right),\quad
A^{\overline{P_3}}(\bar{e},e)=\left(\begin{matrix}\frac{1}{3}&\frac{1}{3}\\
\frac{1}{3}&0\end{matrix}\right),\quad
A^{P_3}(\bar{e},e)=\left(\begin{matrix}0&\frac{1}{3}\\
\frac{1}{3}&\frac{1}{3}\end{matrix}\right).
\]

\begin{figure}[h!]
  \centering
    \includegraphics[width=0.5\textwidth]{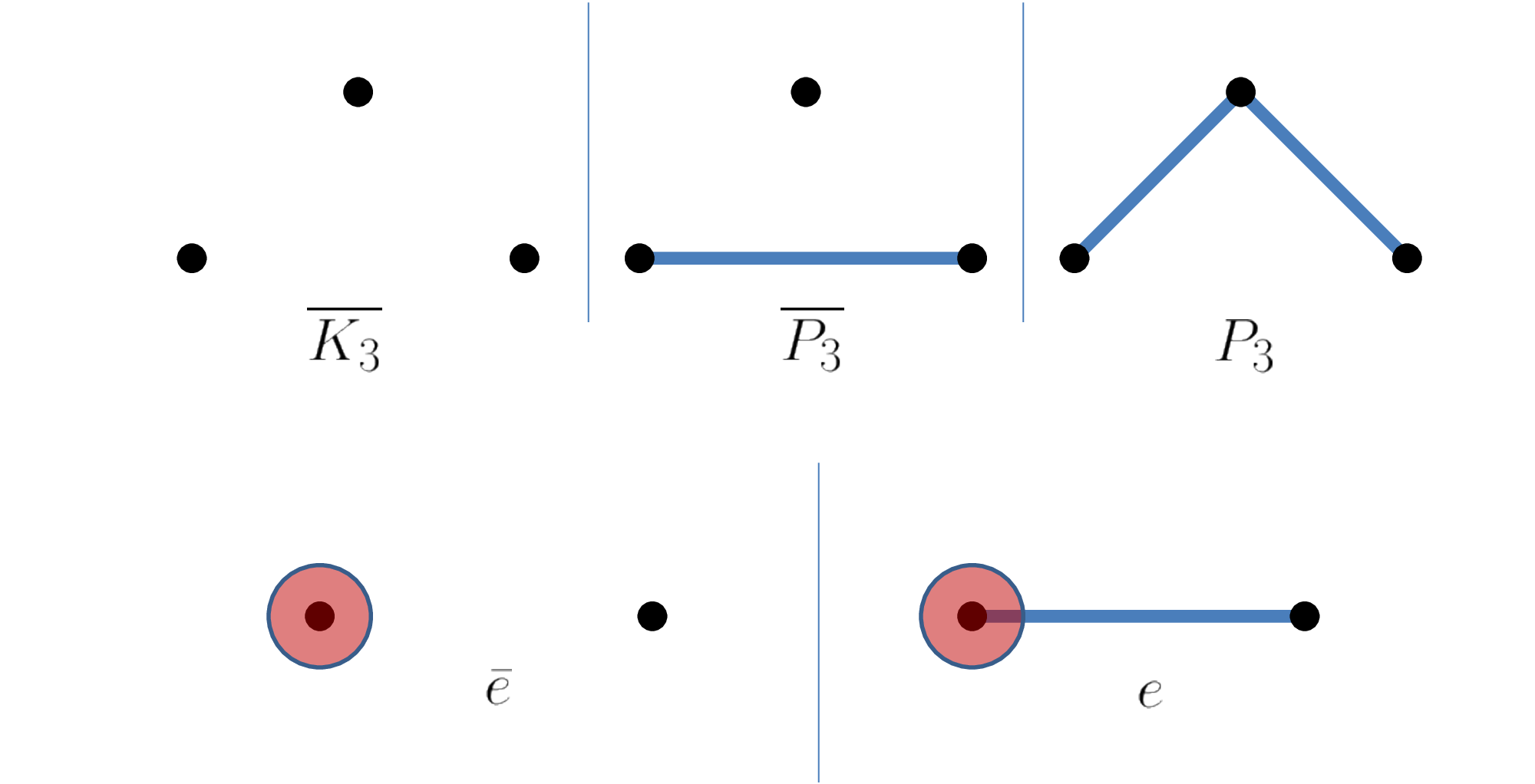}
    \caption{Triangle free graphs and flagged graphs used in the proof}\label{fig:flags}
\end{figure}

Since these are all the graphs on 3 vertices in the family $\mathcal{TF}$, we may apply Corollary \ref{main_cor}, and obtain $(II)$.\\

$(II)\implies (III).$ Suppose $p_0,p_1,p_2$ satisfy the condition in $(II)$. Since $p_0+p_1+p_2=1$, this can be reformulated as
\[
\left(\begin{matrix}3p_0+p_1&1-p_0\\
1-p_0&1-p_0-p_1\end{matrix}\right)\succeq 0,
\]
which implies that
\begin{equation}\label{eqtn:expli_constr}
0 \leq (3p_0+p_1)(1-p_0-p_1)-(1-p_0)^2,
\end{equation}
\[
p_0+p_1\leq 1.
\]
Recall Definition \ref{def_Gxabc} of $G_{x,a,b,c}$, and denote $G_{\alpha,q}:=G_{\alpha,0,0,q}$ a distribution on bipartite graphs, for $\alpha,q\in\I$. Then,
\begin{eqnarray*}
\E [p_0(G_{\alpha,q})]=1-3\alpha(1-\alpha)q(2-q)+o(1).
\end{eqnarray*}
\begin{eqnarray*}
\E[ p_1(G_{\alpha,q})]=6\alpha(1-\alpha)q(1-q)+o(1).
\end{eqnarray*}
By Lemma \ref{lemma:rand_graph}, for every $\alpha, q$, $(\E[p_0(G_{\alpha,q}]),\E[p_1(G_{\alpha,q})],\E[p_2(G_{\alpha,q})])\in\Delta_{\mathcal{BP}}(\overline{K_3},\overline{P_3},P_3)$. Thus, it suffices, given $p_0,p_1$ that satisfy (\ref{eqtn:expli_constr}), to find $(\alpha,q)\in\I^2$ such that,
\[
p_0=1-3\alpha(1-\alpha)q(2-q)\mbox{, ~and~ }
p_1=6\alpha(1-\alpha)q(1-q).
\]
This implies that
\[
q=\frac{2-2p_0-2p_1}{2-2p_0-p_1}\in\I,
\]
and
\[
(1-2\alpha)^2=\frac{(3p_0+p_1)(1-p_0-p_1)-(1-p_0)^2}{3(1-p_0-p_1)}
\]
Miraculously, $\alpha\in\I$ that satisfies this equation exists iff the quadratic constraint in (\ref{eqtn:expli_constr}) are satisfied and $p_0+p_1<1$. Indeed it
is easy to check that in this case the right hand side is non-negative and is $\leq 1$. On the other hand, if
$p_0+p_1=1$, then by (\ref{eqtn:expli_constr}) $p_0=1$ and this profile is attained for $q=0$.
\begin{figure}[H]
  \centering
    \includegraphics[width=0.5\textwidth]{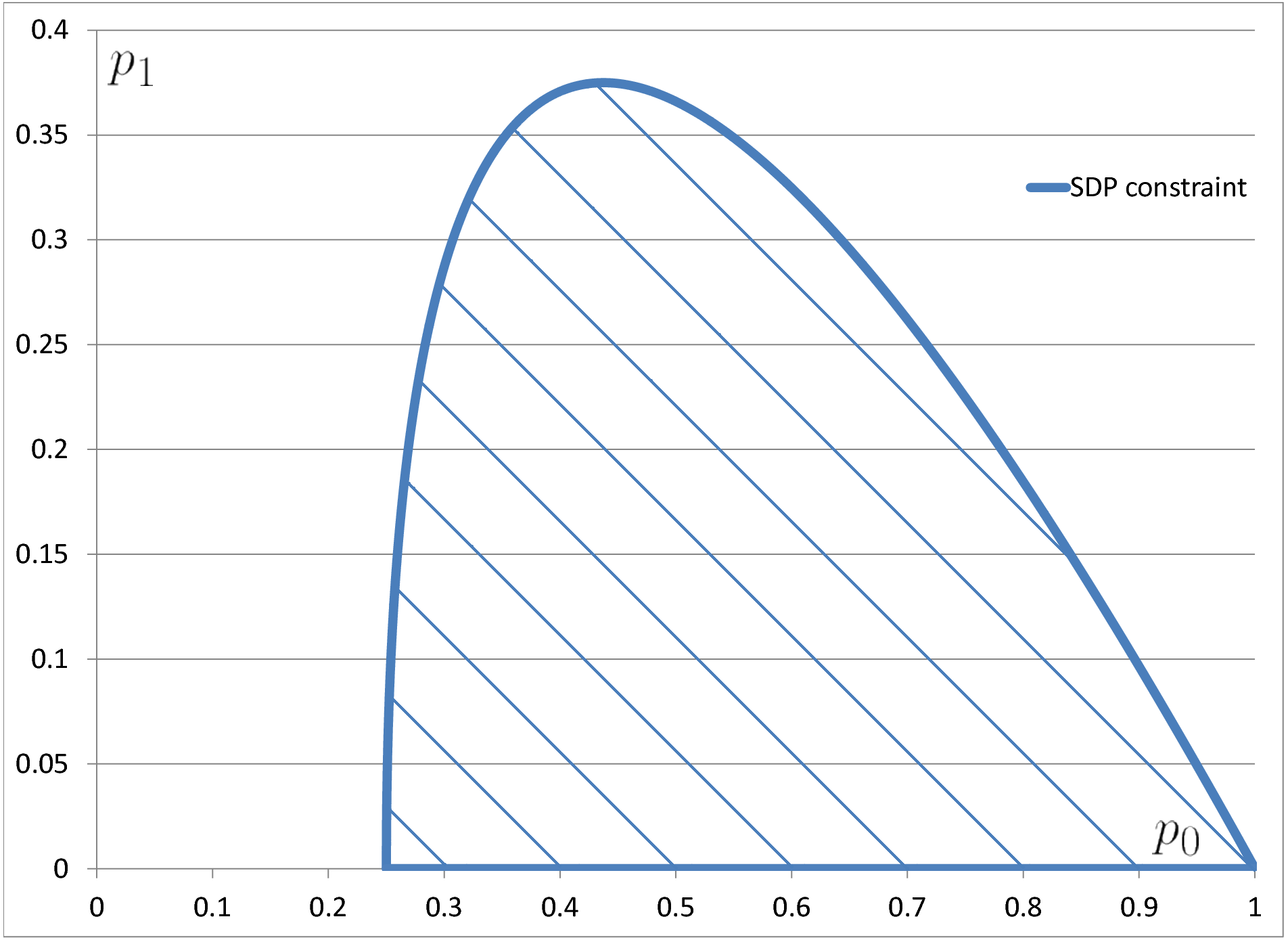}
    \caption{The region of possible $p_0,p_1$ of triangle-free graphs}
\end{figure}

$(III)\implies (I).$ Immediate, since every bipartite graph is triangle free.

\end{proof}

\section{Concluding remarks}
\label{section_concluding}
In this paper we study the set ${\cal S}_3 \subset \mathbb R^4$ of all vectors $(p_0,...,p_3)$ that are arbitrarily close to the 3-local profiles of arbitrarily large
graphs. We show that the projection of this set to the $(p_0,p_3)$ plane is completely realizable by the graphs that are generated by a model which partitions the
vertices into two sets. We also show that the intersection of ${\cal S}_3$ with the plane $p_3=0$, i.e. triangle-free graphs, is completely realizable by a simple
model of random bipartite graphs. We wonder how far these observations can be extended. Razborov's work~\cite{razborov} shows that certain $3$-profiles require the use of
$k$-partite models for arbitrarily large $k$. Also in general, it is not true that a $k$-local profile of every $K_k$-free graph can be realized by
$(k-1)$-partite graph. Indeed, it was shown in \cite{das-et-al}, that already for $k \geq 4$ the minimum density of empty sets of size $k$ in $K_k$-free graphs is strictly smaller than
what can be achieved by $(k-1)$-partite graphs.

It still remains a challenge to get a full description of the set ${\cal S}_3$. The analogous questions concerning $r$-profiles, $r>3$ seem even more difficult. Even
characterizing the profiles of ($r$-cliques, $r$-anticliques), which is solved here for $r=3$, is still open.

\end{document}